\documentclass[10pt]{amsart}

\usepackage{amsmath,amsthm,amsfonts,amssymb}
\usepackage{hyperref}

\newtheorem{theorem}{Theorem}[section]
\newtheorem{lemma}[theorem]{Lemma}
\theoremstyle{definition}
\newtheorem{algorithm}[theorem]{Algorithm}

\newtheorem{proposition}[theorem]{Proposition}

\newtheorem*{theorem*}{Theorem}
\newtheorem*{proposition*}{Proposition}

\newcommand{\rb}[1]{{\left( #1 \right)}}

\newcommand{\rbb}[1]{{\Bigl( #1 \Bigr)}}

\newcommand{\modb}[1]{{\left| #1 \right|}}

\newcommand{\modbb}[1]{{\Bigl| #1 \Bigr|}}

\newcommand{\gp}[1]{{\left\langle #1 \right\rangle}}

\def\s{{\sigma}}

\def\CF{{\mathcal F}}

\def\ovr{{\overline{r}}}

\def\MS{{\mathbb{S}}}
\def\ME{{\mathbb{E}}}
\def\MN{{\mathbb{N}}}
\def\MZ{{\mathbb{Z}}}
\def\MR{{\mathbb{R}}}
\def\BP{{\bf{P}}}

\title[Mean-Set Attack]{Mean-Set Attack: Cryptanalysis of Sibert et al. Authentication Protocol}
\date{\today}

\begin{document}

\author[N. Mosina]{Natalia Mosina}
\address{Department of Mathematics, CUNY/LAGCC, Long Island City, NY, USA}
\email{nmosina@lagcc.cuny.edu}
\thanks{The work of the first author was partially supported by the PSC-CUNY Grant Award 60014-40 41. 
The work of the second author was partially supported by the NSF grant DMS-0914773.}

\author[A. Ushakov]{Alexander Ushakov}
\address{Department of Mathematics, Stevens Institute of Technology, Hoboken, NJ, USA}
\email{sasha.ushakov@gmail.com}

\begin{abstract}
We analyze the Sibert et al. group-based (Feige-Fiat-Shamir type) authentication protocol
and show that the protocol is not computationally zero-knowledge.
In addition, we provide experimental evidence that our approach is practical
and can succeed even for groups with no efficiently computable
length function such as braid groups.
The novelty of this work is that we are not attacking the protocol
by trying to solve an underlying complex algebraic problem, namely,
the conjugacy search problem, but use a probabilistic approach, instead.

\end{abstract}

\maketitle


\section{Introduction}

The group-based cryptography attracted a lot of attention after invention of the
Anshel-Anshel-Goldfeld \cite{AAG} and Ko-Lee et al. \cite{KLCHKP} key-exchange protocols in $1999$.
Since then a number of new cryptographic protocols, including public-key {\em authentication protocols},
based on infinite groups were invented and analyzed. One may consult
\cite{MSU_book} and \cite{Dehornoy_survey} to learn more about general group-based cryptography.
In this paper we consider a particular interactive group-based authentication scheme,
Sibert et al. protocol (see \cite{SDG}, \cite{Dehornoy_survey}).

Recall that
any {\em interactive proof of knowledge system} is
a multi-round randomized protocol for two parties, in which one of the parties
(the Prover) wishes to
convince another party (the Verifier) of the validity of a given assertion.
Every interactive proof of knowledge should satisfy {\em completeness} and {\em soundness}
properties (\cite{Fiat}, \cite{Goldreich_survey:2002}):
\begin{enumerate}
\item[]
{\em Completeness}: If the assertion is true, it should be accepted by the Verifier
with high probability.
\item[]
{\em Soundness}: If the assertion is false, then the Verifier rejects it
with high probability.
\end{enumerate}
If the Prover does not trust the Verifier and does not want to
compromise any private information in the process of providing the proof of identity,
then the following property, concerned with the preservation of security,
becomes very important:
\begin{itemize}
    \item[]
{\em Zero-Knowledge (ZK)}:
Except the validity of the Prover's assertions, no other information is revealed
in the process of the proof.
\end{itemize}
If a given protocol possesses the zero-knowledge property, then it is
considered to be a {\em zero-knowledge interactive proof system}
(\cite{Goldreich_survey:2002}).

There are three different notions of zero-knowledge
that have been commonly used in the
literature (\cite {G}, \cite{Goldreich_survey:2002}, \cite{Dehornoy_survey}); namely,
perfect zero-knowledge, statistical zero-knowledge, and computational zero-knowledge.
The first notion is the most strict definition of ZK, which is rarely useful
in practice.
The last notion of the ZK property (computational zero-knowledge)
is the most liberal notion, and it is
used more frequently in practice than the others.

Sibert et al. authentication protocol,
is an example of an interactive (dynamic, randomized) proof system.
In this paper, we use probabilistic tools,
introduced
in \cite{MosUsh:SLLN1} and outlined in Section \ref{se:SLLN} below, to
design an attack on this particular
cryptographic primitive and show that it is not
computationally zero-knowledge.
In addition, we conduct some experiments that support our conclusions
and show that the protocol is not secure in practice.

\subsection{Description of the protocol}
\label{se:protocol}

The Sibert's protocol is an iterated two-party three-pass
Feige-Fiat-Shamir \cite{Fiat} type authentication protocol.
There are two slightly different descriptions of the protocol
available in \cite{Dehornoy_survey} and \cite{SDG} with two different key generation
algorithms. In \cite{SDG}, the protocol is introduced as \textit{Scheme II}.
Here, we follow
the description of the scheme from the survey \cite{Dehornoy_survey},
except for the minor notational modifications in the conjugation.
These modifications do not affect the protocol and its cryptographic properties
at all (inverting $r$ and $y$ in \cite{Dehornoy_survey} would resolve it). In addition,
\cite{Dehornoy_survey} and \cite{SDG} treat the protocol slightly differently themselves,
with and without a collision-free one-way hash function, respectively.
Nevertheless, it is not essential for our analysis.

Let $G$ be a (non-commutative, infinite)
group, called the {\em platform group} and $\mu$ a probability measure on $G$.
The Prover's {\em private key} is an element $s \in G$,
the Prover's {\em public key} is a pair $(w, t)$, where $w$ is an arbitrary
element of the group $G$, called {\em the base element}, and $t = s^{-1} w s$
is a conjugate of $w$ by $s$.
In addition, we assume that $H$ is a collision-free one-way hash function from $G$ to $\{0,1\}^N$.
A single round of the protocol is performed as follows:
\begin{enumerate}
    \item
The Prover chooses a random element $r \in G$, called the {\em nonce},
according to the probability measure $\mu$,
and sends $x = H(r^{-1} t r)$, called the {\em commitment}, to the Verifier.
    \item
The Verifier chooses a random bit $c$, called the {\em challenge},
and sends it to the Prover.
\begin{itemize}
    \item
If $c=0$, then the Prover sends $y = r$ to the Verifier and the
Verifier checks if the equality $x = H(y^{-1} t y)$ is satisfied.
    \item
If $c=1$, then the Prover sends $y = sr$ to the Verifier and the
Verifier checks if the equality $x = H(y^{-1} w y)$ is satisfied.
\end{itemize}
\end{enumerate}
This round is repeated $k$ times
to guarantee the {\em soundness error} (i.e., probability that a cheating
Prover will be able to convince the Verifier of a false statement)
of order $2^{-k}$, which is considered to be negligible if $k$ is large,
say $k\geq 100$.
The Sibert's protocol satisfies both, completeness and soundness, properties
of interactive proof systems.

In addition, \cite{SDG} describes another
authentication protocol, the so-called \textit{Scheme III}, which is different from the
one described above.
Even though techniques of this paper do not directly
apply to that protocol, we believe that using similar ideas,
this scheme can be successfully attacked as well.

\subsection{Security of the protocol}
\label{se:zk}

Note that if an intruder (named Eve) can compute the secret element $s$ or any element $s' \in G$ such that
$t = s'^{-1} w s'$, i.e., if Eve can solve {\em the conjugacy search problem} for $G$,
then she can authenticate as the Prover. Thus, as indicated in \cite{SDG},
the computational difficulty of the conjugacy search problem for $G$
is necessary for security of this protocol.

Originally, it was proposed to use braid groups $B_n$ (see \cite{Birman_book, Epstein,Kessel_Turaev:2009})
as platform groups,
because there was no efficient solution
of the conjugacy search problem for $B_n$ known.
This motivated a lot of research about braid groups.
As a result of recent developments (\cite{BGG1}, \cite{BGG2}, \cite{BGG3}),
there is an opinion
that the conjugacy search problem for $B_n$ can be solved in polynomial time.
If that is true in fact, then
the Sibert et. al. authentication protocol is insecure for $B_n$.
Nevertheless, the same protocol can be used with other platform groups
and, hence, it is important to have tools for analysis of
this type of general Sibert protocols.
We show in the present paper that it is not necessary
to solve the conjugacy search problem for $G$
to break the scheme.
Instead, one can analyze zero-knowledge property of the protocol by
employing ideas from probability theory and show
that the protocol is insecure
under a mild assumption of existence of an efficiently computable length
function for the platform group $G$.
Even for groups with no efficiently computable length function, such as $B_n$, a reasonable
approximation can do the job.

Now, let $\mu$ be a probability measure on a platform group $G$. We say that $\mu$
is {\em left-invariant} if for every $A\subseteq G$ and $g\in G$ the
equality $\mu(A) = \mu(gA)$ holds.
The following result is proved in \cite{SDG}.

\begin{proposition*}[\cite{SDG}]
{\em Let $G$ be a group. If the conjugacy search problem for $G$ is computationally
hard (cannot be solved by a probabilistic polynomial time Turing machine)
and $\mu$ is a left-invariant probability measure on $G$ then the outlined above protocol
is a zero knowledge interactive proof system.
}
\end{proposition*}

Clearly, there are no left-invariant probability measures on braid groups,
used as platform groups in the protocol,
and, therefore, as noticed in \cite{Dehornoy_survey} and
\cite{SDG}, this protocol cannot be a perfect zero
knowledge interactive proof system when used with an infinite group such as $B_n$.
Nevertheless, it is conjectured in \cite{SDG} that the scheme can be computationally
zero knowledge for certain distributions $\mu$ on $B_n$. The authors
supported that conjecture by statistical arguments based on length analysis.

\subsection{The idea of mean-set attack: the shift search problem}
\label{se:idea}

If we look at the protocol outlined in Section \ref{se:protocol}, we
observe that the Prover sends to the Verifier a sequence
of random elements of two types: $r$ and $sr$, where $r$ is a
randomly generated element and $s$ is the Prover's secret
element. Any passive eavesdropper (Eve) can arrange a table of challenge/response
transactions, where each row corresponds to a single round of the protocol,
as shown below,
\begin{center}\small
\begin{tabular}{|l|p{15mm}|p{25mm}|p{25mm}|}
\hline
  Round & Challenge & Response type \# 1 & Response type \# 2\\
\hline
$1$   & $c=1$ & --    & $sr_1$\\
\hline
$2$   & $c=0$ & $r_2$ & --\\
\hline
$3$   & $c=0$ & $r_3$ & --\\
\hline
$4$   & $c=1$ & --    & $sr_4$\\
\hline
$5$   & $c=0$ & $r_5$ & --\\
\hline
$\ldots$ & $\ldots$ & $\ldots$ & $\ldots$ \\
\hline
$n$ & $c=0$ & $r_{n}$ & --\\
\hline
\end{tabular}
\end{center}
and obtain two sets of elements, corresponding
to $c=0$ and $c=1$ respectively:
    $R_0 = \{r_{i_1},\ldots,r_{i_k}\}$
and
    $R_1 = \{sr_{j_1},\ldots,sr_{j_{n-k}}\},$
where all elements $r_i$ are distributed according to $\mu$,
i.e., all
these elements are generated by the same random generator.
Eve's goal is to recover the secret element $s$ based on the intercepted
sequences $R_0$ and $R_1$. We call this problem a {\em shift search problem}.

To explain the idea of the {\em mean-set attack}, assume for a moment
that the group $G$ is an infinite cyclic group $\MZ$.
In that case, we can rewrite the elements of $R_1$ in additive notation $\{s+r_{j_1},\ldots,s+r_{j_{n-k}}\}$.
Then we can compute the empirical average $\ovr_0 = \frac{1}{k}\sum_{m=1}^k r_{i_m}$
of the elements in $R_0 \subset \MZ$
and the empirical average
$\ovr_1 = \frac{1}{n-k}\sum_{l=1}^{n-k} (s+r_{j_l}) = s+\frac{1}{n-k}\sum_{l=1}^{n-k} r_{j_l}$
of the elements in $R_1 \subset \MZ$.
By the strong law of large numbers for real-valued random variables the larger the
sequence $R_0$ is, the closer the value of $\ovr_0$ to the
actual mean $\ME (\mu)$ of the distribution $\mu$
on $\MZ$, induced by $r$. Similarly, the larger the sequence $R_1$ is,
the closer the value of $\ovr_1$ is to the
number $s+\ME (\mu)$. Therefore, subtracting $\ovr_0$ from $\ovr_1$,
we obtain a good guess of what $s$ is.
Observe three crucial properties that allow us to compute the secret
element in the case $G=\MZ$:
\begin{itemize}
    \item[\bf (AV1)]
(Strong law of large numbers for real-valued random variables) If $\{\xi_i\}_{i=1}^\infty$
is a sequence of independent and identically distribute (i.i.d.) real-valued random variables, then
    $$\frac{1}{n}\sum_{i=1}^n \xi_i {\rightarrow} \ME \xi_1$$
with probability one as $n\rightarrow\infty$, provided $\ME (\xi_1) < \infty.$
    \item[\bf (AV2)]
(''Shift`` property or linearity) For any real-valued random variable $\xi$, the formula
    $$\ME(c + \xi) = c + \ME(\xi)$$
holds.
    \item[\bf (AV3)]
(Efficient computations) The average value $\frac{1}{n}\sum_{i=1}^n \xi_i$
is efficiently computable.
\end{itemize}
Geometrically, we can interpret this approach as follows.
Given a large sample of random, independent,
and identically distributed points $r_{i_1},\ldots,r_{i_k}$
and a large sample of shifted points $s+r_{j_1},\ldots,s+r_{j_{n-k}}$
on the real line,
the shift $s$ is ``effectively visible''.

It turns out that the same is true in general infinite groups.
One can generalize a number of mathematical tools of the classical probability theory
to finitely generated groups (see \cite{MosUsh:SLLN1} and Section \ref{se:SLLN} below)
in order to have the counterparts of (AV1), (AV2), and (AV3).
Indeed,
\begin{itemize}
    \item
for a random group element $\xi:\Omega\rightarrow G$, one can define a set $\ME (\xi) \subseteq G$
called the {\em mean-set},
    \item
for a sample of $n$ random group elements $\xi_1,\ldots,\xi_n$, one can define
their average -- a set $\MS_n = \MS(\xi_1,\ldots,\xi_n) \subseteq G$
called the {\em sample mean-set} of elements $\xi_1,\ldots,\xi_n$,
\end{itemize}
so that we have a ''shift`` property $\ME(s \xi)=s \ME(\xi)$
and a
generalization of the strong law of large numbers (SLLN)
for groups with respect to $\ME (\xi)$ in a sense that
$\MS(\xi_1,\ldots,\xi_n)$ converges to $\ME(\xi_1)$ as
$n\rightarrow\infty$ with probability one (see Section \ref{se:SLLN}
for precise definitions and statements).
In addition, assume that
sample mean $\MS(\xi_1,\ldots,\xi_n)$ is efficiently computable.
Using the operator $\MS$, Eve can compute a set
    $$\MS(sr_{j_1},\ldots,sr_{j_{n-k}}) \cdot [\MS(r_{i_1},\ldots,r_{i_k})]^{-1},$$
which should contain $s$ with high probability when $n$ is sufficiently large.
This is the idea of the mean-set attack and our approach to the {\em shift search problem}.
Furthermore, one can show that the more rounds of the protocol are
performed, the more information about the secret key our attack gains
(note that at the same time the protocol is iterated by its nature,
and large number of rounds is important for its
reliability in a sense of the soundness property).
The discussion above leads to the main theoretical results of this paper,
proved in Section \ref{se:attack}.

\vspace{2mm}\noindent{\bf Theorem A.} {\bf(Mean-set attack principle -- I)}
{\em
Let $G$ be a group, $X$ a finite generating set for $G$, $s\in G$ a secret fixed element,
and $\xi_1, \xi_2, \ldots$ a sequence of randomly generated i.i.d.
group elements,
such that $\ME\xi_1 = \{g\}$. If $\xi_1,\ldots,\xi_n$ is a sample of random elements of $G$
generated by the Prover,
$c_1,\ldots,c_n$ a succession of random bits (challenges) generated
by the Verifier,
and
$$y_i =
    \begin{cases}
    r_i & \mbox{if } c_i=0; \\
    s r_i & \mbox{if } c_i=1\\
    \end{cases}
    $$
    random elements representing responses of the Prover,
then there exists a constant $D=D(G, \mu)$ such that
    $$\BP \rb{s\not\in\MS \rbb{\{y_i \mid c_i=1, i=1,\ldots,n\}} \cdot \MS\rbb{\{y_i \mid c_i=0, i=1,\ldots,n\}}^{-1}} \le \frac{D}{n}.$$
}

\vspace{2mm}\noindent{\bf Theorem B.} {\bf(Mean-set attack principle -- II)}
{\em
If, in addition to the assumptions of Theorem A, the distribution $\mu$
has finite support, then there exists a constant $D=D(G, \mu)$ such that
    $$\BP \rb{s\not\in\MS \rbb{\{y_i \mid c_i=1, i=1,\ldots,n\}} \cdot \MS\rbb{\{y_i \mid c_i=0, i=1,\ldots,n\}}^{-1}} \le O(e^{-Dn}).$$
}

\subsection{Outline}
Section \ref{se:graphs_groups} reviews some necessary
graph- and group-theoretic preliminaries that constitute the
setting of our work. In Section \ref{se:SLLN}, we recall the notion of the mean-set (expectation) of
a (graph-)group-valued random element, introduced in \cite{MosUsh:SLLN1}, and main theorems relevant
to this object to prepare the ground for the main results;
in particular, we discuss the ''shift`` property,
the strong law of large numbers,
and the analogues of Chebyshev and Chernoff-like inequalities for graphs and groups.
In Section \ref{se:mean_set_compute}, we propose an algorithm for computing
mean-sets. Next, we turn to formulations and proofs of the main theoretical results
of this paper, the mean-set attack principles under different assumptions.
This task is carried out in Section \ref{se:attack}.
At the end of that section, we indicate that even if the proposed
algorithm fails, we can still gain some information about the secret
key of the Prover. In other words,
the more rounds of the protocol are
performed, the more information
about the secret key we can gain.
In Section \ref{se:classical_generation}, we present results of our experiments
with the classical key generation according to \cite{Dehornoy_survey}.
Section \ref{se:special_key_generation} is concerned with results of experiments
with the alternative (special) key generation proposed by Sibert et al. in \cite{SDG}.
At the end, in Section \ref{se:defense}, we discuss possible methods for
defending against
the mean-set attack.

\section{Preliminaries}

\label{se:graphs_groups}

Let us briefly recall some definitions of group and graph theory.
For
a better insight into graph theory, the reader is referred to
\cite{West:book2000}, while \cite{Kurosh_book} can serve as a good
introduction into group theory.

\subsection{Graphs}
\label{se:graphs}

An {\em undirected graph} $\Gamma$ is an ordered pair of sets
$(V,E)$ where
\begin{itemize}
    \item
$V = V(\Gamma)$ is called the {\em vertex set};
    \item
$E = E(\Gamma)$ is a set of unordered pairs $(v_1,v_2) \in V\times
V$ called the {\em edge set}.
\end{itemize}
If $e = (v_1,v_2) \in E$ then we say that $v_1$
and $v_2$ are {\em adjacent} in $\Gamma$. The number
of vertices adjacent to $v$ is called the {\em degree} of $v$. We say that the graph $\Gamma$ is {\em locally-finite} if every vertex has a finite degree.

A {\em directed graph} $\Gamma$ is an ordered pair of sets $(V,E)$
where $E = E(\Gamma)$ is a set of ordered pairs $(v_1,v_2) \in
V\times V$. If $e = (v_1,v_2) \in E$, then we say that $v_1$ is the
{\em origin} of the edge $e$, denoted by $o(e)$, and $v_2$ is the
{\em terminus} of $e$, denoted by $t(e)$.
An undirected graph can be viewed as a directed graph in which a pair $(v_1,v_2) \in E$
serves as two edges $(v_1,v_2)$ and $(v_2,v_1)$.

A {\em path} $p$ in a directed graph $\Gamma$ is a finite sequence
of edges $e_1,\ldots,e_n$ such that $t(e_i) = o(e_{i+1})$. The
vertex $o(e_1)$ is called the {\em origin} of the path $p$ and is
denoted by $o(p)$. The vertex $t(e_n)$ is called the {\em terminus}
of the path $p$ and is denoted by $t(p)$. The number $n$ is called
the {\em length} of the path $p$ and is denoted by $|p|$. We say
that two vertices $v_1,v_2\in V(\Gamma)$ are {\em connected}, if
there exists a path from $v_1$ to $v_2$ in $\Gamma$. The graph
$\Gamma$ is {\em connected} if every pair of vertices is connected.

The {\em distance} between $v_1$ and $v_2$ in a graph $\Gamma$ is
the length $d(v_1,v_2)$ of a shortest path between $v_1$ and $v_2$.
If $v_1$ and $v_2$ are disconnected, then $d(v_1,v_2) = \infty$.
We say that a path $p  =e_1,\ldots, e_n$ from $v_1$ to $v_2$ is {\em
geodesic} in a graph $\Gamma$ if $d(o(p),t(p)) = d(v_1,v_2) = n$,
i.e., if $p$ is a shortest path from $v_1$ to $v_2$.

A path $p = e_1,\ldots, e_n$ in a graph $\Gamma$ is {\em closed}, if
$o(p) = t(p)$. In this case we say that $p$ is a {\em cycle} in
$\Gamma$. A path $p$ is {\em simple}, if no proper segment of $p$ is
a cycle. The graph $\Gamma$ is a {\em tree} if it does not contain a
simple cycle.

\subsection{Groups and Cayley graphs}

Consider a finite set, also called {\em alphabet}, $X =
\{x_1,\ldots,x_n\}$, and let $X^{-1}$ be the set of formal inverses
$\{x_1^{-1},\ldots,x_n^{-1}\}$ of elements in $X$. This defines an
involution $^{-1}$ on the set $X^{\pm 1}:= X\cup X^{-1}$ which maps
every symbol $x\in X$ to its formal inverse $x^{-1} \in X^{-1}$ and
every symbol $x^{-1} \in X^{-1}$ to the original $x \in X$. An
alphabet $X$ is called a {\em group alphabet} if $X^{-1}\subseteq X$,
and there is an involution which maps elements of $X$ to their
inverses. An {\em $X$-digraph} is a graph $(V,E)$ with edges labeled
by elements in $X^{\pm 1}= X \cup X^{-1}$ such that for any edge $e
= u \stackrel{x}{\rightarrow} v$ there exists an edge $v
\stackrel{x^{-1}}{\rightarrow} u$, which is called the inverse of
$e$ and is denoted by $e^{-1}$. See \cite{KM} for more information
on $X$-digraphs.

Let $G$ be a group and $X \subset G$ a set of generators for $G$,
i.e. $G = \gp{X}$. Assume that $X$ is closed under inversion, i.e.,
$X = X^{\pm 1}$. The {\em Cayley graph} $C_G(X)$ of $G$ relative to
$X$ is a labeled graph $(V,E)$, where the vertex set is $V=G$, and
the edge set $E$ contains all edges of the form $g_1
\stackrel{x}{\rightarrow} g_2$ where $g_1,g_2\in G$, $x\in X$ and
$g_2 = g_1x$ and only them. The {\em distance} between elements $g_1,g_2 \in G$
relative to the generating set $X$ is the distance in the graph
$C_G(X)$ between vertices $g_1$ and $g_2$ or, equivalently,
    $$d_X(g_1,g_2)=\min\{n \mid g_1 x_1^{\varepsilon_1}x_2^{\varepsilon_2} \ldots x_n^{\varepsilon_n}=g_2 \mbox{ for some } x_i\in X, \varepsilon_i=\pm 1\}.$$

\subsection{Random (graph-)group elements}
\label{se:SLLN}

In this section, we recall some of the main notions
and results of \cite{MosUsh:SLLN1} that are
employed further in the present paper.
Let $\Gamma = (V,E)$ be a locally-finite connected graph and
$(\Omega,\CF,P)$ a probability space.
A measurable mapping $\xi:\Omega \rightarrow V(\Gamma)$ is called
a {\em random graph element} defined on a given
probability space.
A random $\Gamma$-element $\xi$ induces an atomic probability measure $\mu$ on $V(\Gamma)$
defined in a usual way as
    $$\mu(v) = \mu_\xi(v) = \BP\{\omega \mid \xi(\omega)= v \}, ~ v\in V(\Gamma).$$
Define a {\em weight function} $M_\xi:V(\Gamma) \rightarrow \MR$ by
    $$M(v) = M_\xi(v) = \sum_{s\in V(\Gamma)} d^2(v,s) \mu_\xi(s),$$
where $d(v,s)$ is the distance between $v$ and $s$ in $\Gamma$. The domain
of $M$ is the set
    $$domain(M) = \{v\in\ V(\Gamma) \mid \sum_{s\in V(\Gamma)} d^2(v,s) \mu_\xi(s)<\infty \}.$$
It is proved in \cite{MosUsh:SLLN1} that for any distribution $\mu$ on $V(\Gamma)$ either
$domain(M) = \emptyset$ or $domain(M) = V(\Gamma)$.
In the case when $domain(M) = V(\Gamma)$, we say that $M(\cdot)$ is {\em totally defined}.
Given that $domain(M) = V(\Gamma)$, the {\em mean-set} of a $\Gamma$-valued $\xi$
is defined to be a set of vertices minimizing the weight function, i.e.,
\begin{equation}\label{eq:mean-set}
    \ME (\xi) = \{v \in V(\Gamma) \mid  M(v) \le M(u), ~~ \forall u\in V(\Gamma) \}.
\end{equation}
Sometimes we write $\ME(\mu)$ and speak of the mean-set of distribution
$\mu$.
Using the Cayley graph construction one can similarly define a notion of the mean-set
for a finitely generated group $G$
(relative to a fixed generating set).
Similar mean values (in different settings)
are used rather often; see \cite{MosUsh:SLLN1} for some history and literature sources.
Below, we recall some results proved in \cite{MosUsh:SLLN1}.

\begin{lemma}[\cite{MosUsh:SLLN1}]\label{le:E_finite}
Let $\xi$ be a random $\Gamma$-element, where $\Gamma$ is a connected locally-finite graph, with
totally defined weight function $M_\xi(\cdot)$. Then the mean-set
$\ME(\xi)$ is non-empty and finite.
\end{lemma}

The next property is an analogue of the property $\ME(c+\xi) = c+\ME\xi$
for real-valued random variables.

\begin{proposition}[Shift property, \cite{MosUsh:SLLN1}]
\label{pr:g_shift}
Let $G = \gp{X}$ be a finitely generated group and $g \in G$.
Let $\xi$ be a random $G$-element. Then for a
random element $\xi_g$ defined by $\xi_g(\omega) := g\xi(\omega)$ we
have $\ME (\xi_g) = g\ME (\xi).$
\end{proposition}
\noindent It is easy to see that this property
follows from the fact that for any
$g_1,g_2,s \in G$ the equality
$d_X(g_1,g_2) = d_X(sg_1,sg_2)$ holds, where $d_X(g_1,g_2)$ is
the {\em distance} between elements $g_1,g_2 \in G$
relative to $X$ (see Section \ref{se:graphs_groups}).

Now let $\xi_1, \ldots, \xi_n$ be a sample of independent and
identically distributed graph-valued random elements
$\xi_i:\Omega\rightarrow V(\Gamma)$ defined on a given probability
space $(\Omega,\CF,\BP)$ and
$\mu_n(v)$ be the relative frequency
   $$ \mu_n(v)= \mu_n(v,\omega) = \frac{|\{i\mid \xi_i(\omega)=v,~~ 1\le i\le n\}|}{n}$$
with which the value $v\in V(\Gamma)$ occurs in the random sample
$\xi_1(\omega), \ldots, \xi_n(\omega)$. Let $$M_{n}(v) =
\sum_{i\in V(\Gamma)} d^2(v,i) \mu_n(i)$$ be the random weight,
called the {\em sampling weight}, corresponding to $v\in V(\Gamma)$,
and $M_{n}(\cdot)$ the resulting random {\em sampling weight
function}.
The set of vertices
    $$\MS_n = \MS(\xi_1,\ldots,\xi_n)= \{v \in V(\Gamma) \mid  M_n(v) \le M_n(u), ~~ \forall u\in V(\Gamma)\}$$
is called the {\em sample mean-set} (or {\em sample center-set})
relative to $\xi$.
The next theorem shows that the sets $\MS_n$ and $\ME (\xi)$ in $\Gamma$
play roles analogous  to the classical average of real values $\frac{x_1+\ldots+x_n}{n}$ and the classical
expectation $\ME$ of a real-valued random variable respectively,
in the non-commutative case. In other words, the strong law of large
numbers generalized to graphs and groups states that our (empirical) sample mean-set
$\MS_n$ converges to the
(theoretical) mean-set $\ME(\xi)$ as $n\rightarrow\infty$.

\begin{theorem}[Strong law of large numbers, \cite{MosUsh:SLLN1}]
\label{th:SLLN}
Let $\Gamma$ be a locally-finite connected graph and $\{\xi_i\}_{i=1}^\infty$ a
sequence of i.i.d. random $\Gamma$-elements. If the weight function $M_{\xi_1}(\cdot)$
is totally defined and
    $\ME(\xi_1) = \{v\}$
for some $v\in V(\Gamma)$, then
    $$\lim_{n\rightarrow\infty} \MS_n = \ME(\xi_1)$$
with probability one.
\end{theorem}

Similar result holds for multi-vertex mean-sets. See \cite{MosUsh:SLLN1}
for technical conditions needed, as well as other details.
The simplest
version of multi-vertex SLLN in terms of {\em limsup} is as follows:

\begin{theorem}[Multi-Vertex SLLN, \cite{MosUsh:SLLN1}]
\label{th:SLLN2}
Let $\Gamma$ be a locally-finite connected graph and
$\{\xi_i\}_{i=1}^\infty$ be a sequence of i.i.d. random $\Gamma$-elements.
Assume that the weight function $M_{\xi_1}(\cdot)$ is totally defined and $\ME(\xi) = \{v_1,\ldots,v_k\}$, where $k\ge 4$.
If $\ME (\xi_1) \subseteq supp(\mu)$ then
    $$\limsup_{n\rightarrow \infty}\MS_n = \ME(\xi_1)$$
holds with probability one.
\end{theorem}

Moreover, the following asymptotic upper bounds
(analogues of the classical Chebyshev and Chernoff bounds) on convergence rate hold:

\begin{theorem}[Chebyshev's inequality for graphs, \cite{MosUsh:SLLN1}]
\label{th:chebyshev}
Let $\Gamma$ be a locally-finite connected graph and
$\{\xi_i\}_{i=1}^\infty$ a sequence of i.i.d. random $\Gamma$-elements.
If the weight function $M_{\xi_1}(\cdot)$ is totally defined
then there exists a constant $C = C(\Gamma,\xi_1)>0$ such that
\begin{equation}\label{eq:Chebyshev2}
    \BP( \MS(\xi_1,\ldots,\xi_n) \not\subseteq \ME(\xi_1) ) \le \frac{C}{n}.
\end{equation}
\end{theorem}

With an additional assumption on $\mu$, we can get even Chernoff-like asymptotic
bound.

\begin{theorem}[Chernoff-like bound for graphs, \cite{MosUsh:SLLN1}]\label{th:Hoeffding}
Let $\Gamma$ be a locally-finite connected graph and
$\{\xi_i\}_{i=1}^\infty$ a sequence of i.i.d. random $\Gamma$-elements.
If the weight function $M_{\xi_1}(\cdot)$ is totally defined
and $\mu_{\xi_1}$ has finite support, then for some constant $C>0$
\begin{equation}\label{eq:Hoeffding2}
    \BP\rbb{ \MS(\xi_1,\ldots, \xi_n) \not\subseteq \ME(\xi_1) } \le O(e^{-Cn}).
\end{equation}
\end{theorem}

\section{Effective computation of a mean-set}
\label{se:mean_set_compute}

Let $G$ be a group and $\{\xi_i\}_{i=1}^n$ a sequence of random i.i.d. elements taking values in $G$
such that the corresponding weight function $M(\cdot)$ is totally defined.
In Section \ref{se:SLLN}, we introduced a notion of the mean-set of $\xi$ that satisfies
the desirable properties (AV1) and (AV2) of Section \ref{se:idea}.
One of the technical difficulties encountered in practice is that,
unlike the classical average value $(x_1+\ldots+x_n)/n$
for real-valued random variables, the sample mean-set $\MS_n$ is hard to compute.
In other words, in general, our definition of the meat-set might not satisfy the
property (AV3).

Several problems arise when trying to compute $\MS_n$:
\begin{itemize}
    \item
Straightforward computation of the set $\{M(g) \mid g\in G\}$ requires at least $O(|G|^2)$
steps. This is computationally infeasible for large groups $G$, and
impossible for infinite groups. Hence  we might want to
reduce the search of a minimum to some small part of $G$.
    \item
There exist infinite groups in which the distance function $d(\cdot,\cdot)$
is very difficult to compute. The braid group $B_\infty$ is an example for such a group.
The computation of the distance function for $B_\infty$
is known to be NP-hard, see \cite{PRaz}. Such groups require special
treatment.

\noindent Moreover, there exist infinite groups for which the distance
function $d(\cdot,\cdot)$ is not computable. We omit consideration of such
groups.
\end{itemize}
We devise a heuristic procedure to solve the first problem.
As proved in \cite{MosUsh:SLLN1}, if the weight function $M(\cdot)$ satisfies certain
local monotonicity properties, then our procedure achieves the desired result.
Our algorithm is a simple direct descent heuristic, in which we
use the sample weight function $M_n$ that comes from
a sample of random group elements $\{g_1,\ldots,g_n\}$ from a finitely-generated group $G$.

\begin{algorithm}[Direct Descent Heuristic] \label{al:direct_descend}\
\\{\sc Input:} A group $G$ with a finite set of generators $X\subseteq G$ and
a sequence of elements $\{g_1,\ldots,g_n\}$ in $G$.
\\{\sc Output:} An element $g \in G$ that locally minimizes $M_n(\cdot)$.
\\{\sc Computations:}
\begin{itemize}
    \item[A.]
Choose a random $g \in G$ according to some probability measure $\nu$ on $G$.
    \item[B.]
If for every $x\in X^{\pm 1}$, $M_n(g) \le M_n(gx)$, then output $g$.
    \item[C.]
Otherwise put $g \leftarrow gx$, where $x\in X^{\pm 1}$ is an element
minimizing the value of $M_n(gx)$ and go to step B.
\end{itemize}
\end{algorithm}

As any other direct descend heuristic method, Algorithm \ref{al:direct_descend}
might not work if the function $M_n$ has local minima. It is proved in \cite{MosUsh:SLLN1}
that it always works for trees and, hence, for free groups.

\begin{theorem}[\cite{MosUsh:SLLN1}] \label{thm:Algorithm_works_trees}
Let $\mu$ be a distribution on a locally-finite tree $T$ such that a
function $M$ is totally defined. Then Algorithm
\ref{al:direct_descend} for $T$ and $M$ finds a central point
(mean-set) of $\mu$ on $T$.
\end{theorem}

The second problem of computing $\MS_n$ concerns practical
computations of length function in $G$.
It turns out that we need a relatively mild assumption to deal with it
-- the existence of an efficiently computable distance
function $d_X(\cdot,\cdot)$; even a ``reasonable'' approximation of the length function
may work.
In this work we approximate geodesic length using the method
described in \cite{MSU_CRYPTO}. Even though it does not guarantee the
optimal result, it was proved to be practically useful in a series of
attacks, see \cite{MSU_PKC,MU_AAGL,MU3,Longrigg_Ushakov2007}.

\section{The mean-set attack}
\label{se:attack}

In this section, we use theoretical results stated above
to attack the Sibert et al. protocol, described in Section \ref{se:protocol}.
In the following heuristic attack we use the Algorithm \ref{al:direct_descend}
to compute sample mean-set $\MS_n$.

\begin{algorithm}{\bf (The mean-set attack)} \label{al:attack}
\\{\sc Input:} The Prover's public element $(t,w)$ and sequences $R_0$ and $R_1$ as in the protocol.
\\{\sc Output:} An element $z$ satisfying the equality $t = z^{-1} w z$ (which can be considered as the Prover's
private key), or $Failure$.
\\{\sc Computations:}
\begin{itemize}
    \item[A.]
Apply Algorithm \ref{al:direct_descend} to $R_0$ and obtain $g_0$.
    \item[B.]
Apply Algorithm \ref{al:direct_descend} to $R_1$ and obtain $g_1$.
    \item[C.]
If $g_1 g_0^{-1}$ satisfies $t = (g_1 g_0^{-1})^{-1} w (g_1 g_0^{-1})$ then output $g_1 g_0^{-1}$. Otherwise output $Failure$.
\end{itemize}
\end{algorithm}

If the algorithm outputs an element $z \in G$, then $z$ can serve
as the Prover's original secret $s$;
any solution of the conjugacy equation $t = x^{-1} w x$ does.
In general, $z$ can be different from $s$, and there are no means for
the adversary to determine whether $z = s$.
In spite of that, Eve, who is only
trying to authenticate as the Prover,
considers this $z$ a success.
On the other hand, since our goal is to show that the protocol is not
computationally zero-knowledge,
we estimate the probability to find $s$. Only this original secret element
$s$ is considered as a success in our analysis.
Other outcomes that work for Eve (when $z\ne s$) are ignored.

The theorems below give asymptotic bounds on the failure rate
(for the original $s$) in the mean-set attack.
We show that the probability of the failure can decrease
linearly or exponentially, depending on the distribution $\mu$.

\begin{theorem}[Mean-set attack principle -- I]
\label{th:mean_attack}
Let $G$ be a group, $X$ a finite generating set for $G$, $s\in G$ a secret fixed element,
and $\xi_1, \xi_2, \ldots$ a sequence of randomly generated i.i.d.
group elements,
such that $\ME(\xi_1) = \{g\}$. If $\xi_1,\ldots,\xi_n$ is a sample of random elements of $G$
generated by the Prover,
$c_1,\ldots,c_n$ a succession of random bits (challenges) generated
by the Verifier,
and
$$y_i =
    \begin{cases}
    r_i & \mbox{if } c_i=0; \\
    s r_i & \mbox{if } c_i=1\\
    \end{cases}
    $$
    random elements representing responses of the Prover,
then there exists a constant $D$ such that
    $$\BP \rb{s\not\in\MS \rbb{\{y_i \mid c_i=1, i=1,\ldots,n\}} \cdot \MS\rbb{\{y_i \mid c_i=0, i=1,\ldots,n\}}^{-1}} \le \frac{D}{n}.$$
\end{theorem}

\begin{proof}
It follows from Theorem \ref{th:chebyshev} that there exists a constant $C$
such that
    $$\BP( \MS(\{y_i \mid c_i=0,i=1,\ldots,n\}) \ne \{g\} ) \le \frac{C}{|\{i\mid c_i=0,i=1,\ldots,n\}|}.$$
Applying Chebyshev's inequality to Bernoulli random variables $\{c_i\}$
having $\ME(c_i)= \frac{1}{2}$ and $\sigma_{c_i}^2=\frac{1}{4}$,
we obtain
    $$\BP\rb{ |\{i\mid c_i=0,i=1,\ldots,n\}|<\frac{n}{4} } < \frac{4}{n}.$$

    In more detail, if number of zeros in our sample of challenges is less than $\frac{n}{4}$, then
the number of ones is greater or equal to $\frac{3n}{4}$, and we have
$$\BP\rbb{ \modbb{\{i\mid c_i=0, i=1,\ldots,n\}} < \frac{n}{4} } < \BP\rb{\modb{ \sum_{i=1}^n c_i - \frac{n}{2}} \geq \frac{n}{4}}.$$
Note that
$$\modb{ \sum_{i=1}^n c_i - \frac{n}{2}}\geq \frac{n}{4} \Leftrightarrow \modb{ \frac{\sum_{i=1}^n c_i}{n} - \frac{1}{2}}
\geq \frac{1}{4} $$
and
$$\BP \rb{ \modb{ \frac{\sum_{i=1}^n c_i}{n} - \frac{1}{2}}
\geq \frac{1}{4}} \leq \frac{4}{n}$$
from the classical Chebyshev inequality for sample means with $\varepsilon = \frac{1}{4}$.

It follows that
    $$\BP( \MS(\{y_i \mid c_i=0,i=1,\ldots,n\}) \ne \{g\} ) \le \frac{4}{n} + \frac{4C}{n} \le \frac{4+4C}{n}.$$
Similarly, we prove that
    $\BP( \MS(\{y_i \mid c_i=1,i=1,\ldots,n\}) \ne \{s g\} ) \le \frac{4+4C}{n}.$
Hence,
    $$\BP(s\not\in\MS(\{y_i \mid c_i=1,i=1,\ldots,n\}) \cdot \MS(\{y_i \mid c_i=0,i=1,\ldots,n\})^{-1}) \le \frac{8+8C}{n}.$$
\end{proof}

Furthermore, we can get Chernoff-like asymptotic bound if we impose one restriction
on distribution $\mu$. Recall the original
Hoeffding's inequality (\cite{Hoeffding:1963}) well-known in probability theory.
Assume that $\{x_i\}$ is a sequence of independent random variables and
that every $x_i$ is almost surely bounded, i.e.,
$\BP(x_i-\ME x_i \in[a_i,b_i]) = 1$ for some $a_i,b_i\in\MR$. Then for the
sum $S_n = x_1+\ldots+x_n$, the inequality
    $$\BP(S_n-\ME S_n \ge n\varepsilon) \le \exp \rb{-\frac{2n^2\varepsilon^2}{\sum_{i=1}^n (b_i-a_i)^2}}$$
holds. If $x_i$ are identically distributed, then we get the inequality
\begin{equation}\label{eq:Hoeffding}
    \BP\rb{\frac{1}{n}(x_1+\ldots+x_n) - \ME x_1 \ge \varepsilon} \le 2 \exp\rb{-\frac{2\varepsilon^2}{(b-a)^2}n}.
\end{equation}

Now we can prove the Mean-set attack principle with exponential bounds.
\begin{theorem}[Mean-set attack principle -- II]
Let $G$ be a group, $X$ a finite generating set for $G$, $s\in G$ a secret fixed element,
and $\xi_1, \xi_2, \ldots$ a sequence of randomly generated i.i.d.
group elements,
such that $\ME(\xi_1) = \{g\}$. If $\xi_1,\ldots,\xi_n$ is a sample of random elements of $G$
generated by the Prover,
$c_1,\ldots,c_n$ a succession of random bits (challenges) generated
by the Verifier,
$$y_i =
    \begin{cases}
    r_i & \mbox{if } c_i=0; \\
    s r_i & \mbox{if } c_i=1\\
    \end{cases}
    $$
    random elements representing responses of the Prover,
    and the distribution $\mu$ has finite support,
     then there exists a constant $D=D(G, \mu)$ such that
    $$\BP \rb{s\not\in\MS \rbb{\{y_i \mid c_i=1, i=1,\ldots,n\}} \cdot \MS\rbb{\{y_i \mid c_i=0, i=1,\ldots,n\}}^{-1}} \le O(e^{-Dn}).$$
\end{theorem}

\begin{proof}
It follows from Theorem \ref{th:Hoeffding} that there exists a constant $C$
such that
    $$\BP( \MS(\{y_i \mid c_i=0,i=1,\ldots,n\}) \ne \{g\} ) \le O(e^{-C|\{i\mid c_i=0,i=1,\ldots,n\}|}).$$
Applying inequality (\ref{eq:Hoeffding}) to Bernoulli random variables $\{c_i\}$, we get
    $$\BP\rb{ \sum_{i=1}^n c_i -\frac{1}{2}>\frac{1}{4} } < e^{-n/8}.$$
Thus, we obtain a bound
    $$\BP( \MS(\{y_i \mid c_i=0,i=1,\ldots,n\}) \ne \{g\} ) \le e^{-n/8} + O(e^{-Cn/4}).$$
Similarly, we prove that
    $\BP( \MS(\{y_i \mid c_i=1,i=1,\ldots,n\}) \ne \{s g\} ) \le e^{-n/8} + O(e^{-Cn/4}).$
Hence,
    $$\BP(s\not\in\MS(\{y_i \mid c_i=1,i=1,\ldots,n\}) \cdot \MS(\{y_i \mid c_i=0,i=1,\ldots,n\})^{-1}) \le O(e^{-Dn})$$
where $D = \min\{1/8,C/4\}$.
\end{proof}

Algorithm \ref{al:attack} can fail. Nevertheless
the pair of the obtained elements $g_0,g_1$ often encodes some additional information about
the secret $s$.
Indeed, assume that $\ME\mu = \{g\}$. The element $g_0$ obtained at step A of
Algorithm \ref{al:attack} can be viewed as a product
$g e_0$ for some $e_0\in G$. Similarly, the element $g_1$ can be viewed as a product
$sg e_1$ for some $e_1 \in G$.
Hence Algorithm \ref{al:attack} outputs the secret element $s$ whenever
$g_1g_0^{-1} = s g e_1 e_0^{-1} g^{-1} = s$, i.e., whenever $e_1 e_0^{-1} = 1$.

Now, assume that Algorithm \ref{al:attack} has failed, i.e., $e_1 e_0^{-1} \ne 1$.
In this case, one can try to reconstruct the secret element $s$ as a product
    $$g_1 \cdot e \cdot g_0^{-1} = s g e_1 \cdot e \cdot e_0^{-1} g^{-1}$$
where $e$ is an unknown element of the platform group. Clearly, $e$ gives a correct
answer if and only if $e_1 \cdot e \cdot e_0^{-1} = 1$ or $e = e_1^{-1}e_0$.
The element
\begin{equation}\label{eq:error}
    e_1^{-1}e_0
\end{equation}
is called {\em the error of the method}. Clearly,
one only needs to enumerate all words $e$ of length
up to $|e_1^{-1}e_0|$ to reconstruct the required $s$ in the form $g_1 e g_0^{-1}$.
If a secret element $s$ is chosen uniformly as a word of length $l$
and $|e_1^{-1}e_0|<l$, then we gain some information about $s$, since the search space for $s$ reduces.
We can improve Algorithm \ref{al:attack} by adding such enumeration step as follows.

\begin{algorithm}{\bf (The attack--2)} \label{al:attack2}
\\{\sc Input:} The Prover's public element $(t,w)$. Sequences $R_0$ and $R_1$ as in the protocol.
The number $k\in\MN$ -- the expected length of error element $e_1e_0^{-1}$.
\\{\sc Output:} An element $z$ satisfying the equality $t = z^{-1} w z$
(which can be considered as the Prover's private key), or $Failure$.
\\{\sc Computations:}
\begin{itemize}
    \item[A.]
Apply Algorithm \ref{al:direct_descend} to $R_0$ and obtain $g_0$.
    \item[B.]
Apply Algorithm \ref{al:direct_descend} to $R_1$ and obtain $g_1$.
    \item[C.]
For every word $e$ of lengths up to $k$, check if $g_1 e g_0^{-1}$
satisfies the equality $t = (g_1 e g_0^{-1})^{-1} w (g_1 e g_0^{-1})$ and if so output $g_1 e g_0^{-1}$.
Otherwise output $Failure$.
\end{itemize}
\end{algorithm}

\section{Experiments}
\label{se:appl_SD}

To demonstrate the practical use of our mean-set attack, we perform a series
of experiments, which we describe below. In \cite{SDG}, \cite{Dehornoy_survey}
two different methods of generation of nonce elements were proposed, both
with the same platform group $B_n$, which
has the following (Artin's) presentation
\begin{displaymath}
B_n =
 \left\langle
\begin{array}{lcl}\s_1,\ldots,\s_{n-1} & \bigg{|} &
\begin{array}{ll}
 \s_i\s_j\s_i=\s_j\s_i\s_j & \textrm{if }|i-j|=1 \\
 \s_i\s_j=\s_j\s_i & \textrm{if }|i-j|>1
\end{array}
\end{array}
\right\rangle.
\end{displaymath}

We distinguish between the two ways, classical (\cite{Dehornoy_survey})
and alternative (\cite{SDG}),
to generate elements of the
underlying group by performing two different sets of experiments
outlined below in Sections \ref{se:classical_generation}
and \ref{se:special_key_generation}.
In both cases, we observe that the secret information
of the Prover is not secure, and the probability to break the protocol
grows as the number of rounds of the protocol increases.
All experiments are done using the CRAG software
package \cite{CRAG}.

\subsection{Classical key generation}
\label{se:classical_generation}
Classical key generation of the elements of $B_n$
was suggested in \cite{Dehornoy_survey}
with parameters $n=50$ (rank of the braid group) and the lengths of private keys $L = 512$.
The length function relative to the Artin generators
$\{\s_1,\ldots,\s_{n-1}\}$ is $NP$-hard.
That is why in this paper, as it was already mentioned in
Section \ref{se:mean_set_compute}, we use the approximation of geodesic length
method,
proposed in \cite{MSU_PKC}. See \cite{MSU_PKC,MU_AAGL,MU3,Longrigg_Ushakov2007} for a series of
successful attacks using this method.
We want to emphasize that we compute the sampling weight values in the
Algorithm \ref{al:direct_descend},
which is a subroutine in Algorithm \ref{al:attack},
using the approximated distance function values in $B_n$.

One of the disadvantages of the approximation algorithm that we used is that
there is no polynomial time upper bound for that as it uses Dehornoy
handle-free forms \cite{Dehornoy_fast}. As a result we do not know the complexity of our algorithm
and we do not know how our algorithm scales with parameter values.
In each experiment we randomly generate an instance
of the authentication protocol and try to break it, i.e., find the
private key, using the techniques developed in this paper. Recall
that each authentication is a series of $k$ $3$-pass
commitment-challenge-response rounds. Therefore, an instance of
authentication consists of $k$ triples $(x_i,c_i,r_i)$, $i=1,\ldots,k$ obtained as described
in Section \ref{se:protocol}. Here $x_i$ is a commitment, $c_i$ is a challenge,
and $r_i$ is a response.
 A random bit $c_i$ is
chosen randomly and uniformly from the set $\{0,1\}$. In our
experiments we make an assumption that exactly half of $c_i$'s are
$0$ and half are $1$. This allows us to see an instance of the protocol as a pair
of equinumerous sets $R_0 = \{r_1,\ldots, r_{k/2}\} \subset B_n$ and
$R_1 = \{sr_1',\ldots,sr_{k/2}'\} \subset B_n$.

The main parameters for the system are the rank $n$ of the braid group,
the number of rounds $k$ in the protocol, and the length $L$ of secret keys. We generate a single
instance of the problem with parameters $(n,k,L)$ as follows:
\begin{itemize}
    \item
A braid $s$ is chosen randomly and uniformly as a word of length $L$
over a group alphabet $\{\s_1,\ldots,\s_{n-1}\}$. This braid is a
secret element which is used only to generate further data and to
compare the final element to.
    \item
A sequence $R_0 = \{r_1,\ldots,r_{k/2}\}$ of braid words chosen randomly and
uniformly as words of length $L$ over a group alphabet
$\{\s_1,\ldots,\s_{n-1}\}$.
    \item
A sequence $R_1 = \{sr_1',\ldots,sr_{k/2}'\}$ of braid words, where $r_i'$
are chosen randomly and uniformly as words of length $L$ over a
group alphabet $\{\s_1,\ldots,\s_{n-1}\}$.
\end{itemize}
For every parameter set $(n,k,L)$ we generate $1000$ random instances
$(R_0,R_1)$ and run Algorithm \ref{al:attack} which
attempts to find the secret key $s$ used in the generation of $R_1$.

Below we present the results of actual experiments done for groups
$B_5$, $B_{10}$, and $B_{20}$. Horizontally we have increasing
number of rounds $k$ from $10$ to $320$ and vertically we have
increasing lengths $L$ from $10$ to $100$.
Every cell contains a pair $(P\%,E)$
where $P$ is a success rate and $E$ is an average length of the error (\ref{eq:error})
of the method for the corresponding pair $(L,k)$ of parameter values.
All experiments were performed using CRAG library \cite{CRAG}.
The library provides an environment to test cryptographic protocols constructed
from non-commutative groups, for example the braid group.

\begin{table}[ht]\small
  \centering
  \begin{tabular*}{1\textwidth}%
     {@{\extracolsep{\fill}}|c||c|c|c|c|c|c|}
 \hline
 {\bf L$\backslash$k} &{\bf 10} & {\bf 20} & {\bf 40} & {\bf 80} & {\bf 160} & {\bf 320} \\
 \hline
{\bf  10}& (19\%,~1.3)&  (72\%,~0.3) &  (97\%,~0.04)&  (100\%,~0)& (100\%,~0)& (100\%,~0)\\
 \hline
{\bf  50}& (2\%,~13.4)& (8\%,~9)& (68\%,~1.3)& (93\%,0.1)& (100\%,~0)& (100\%,~0)\\
 \hline
{\bf 100}& (0\%,~53.7)& (0\%,~48.1) & (6\%,~26.9)& (44\%,~14)& (65\%,~14.7)& (87\%,~5)\\
 \hline
  \end{tabular*}
  \caption{Experiments in $B_5$.}
  \label{tb:braid5}
  \end{table}
  \begin{table}[ht]\small
  \centering
  \begin{tabular*}{1\textwidth}%
     {@{\extracolsep{\fill}}|c||c|c|c|c|c|c|}
 \hline
 {\bf L$\backslash$k} &{\bf 10} & {\bf 20} & {\bf 40} & {\bf 80} & {\bf 160} & {\bf 320} \\
 \hline
{\bf  10}& (15\%,~1.8)&  (68\%,~0.3)& (98\%,~0)& (100\%,~0)& (100\%,~0)& (100\%,~0)\\
 \hline
{\bf  50}& (0\%,~4.5)&   (23\%,~1.3)& (82\%,~0)&  (97\%,~0)& (99\%,~0)& (100\%,~0)\\
 \hline
{\bf 100}& (1\%,~41)& (7\%~,23.5)& (33\%,5)& (79\%,~1)& (97\%,~0.6)& (98\%,~1.1)\\
 \hline
  \end{tabular*}
  \caption{Experiments in $B_{10}$.}
  \label{tb:braid10}
  \end{table}
  \begin{table}[ht]\small
  \centering
  \begin{tabular*}{1\textwidth}%
     {@{\extracolsep{\fill}}|c||c|c|c|c|c|c|}
 \hline
 {\bf L$\backslash$k} &{\bf 10} & {\bf 20} & {\bf 40} & {\bf 80} & {\bf 160} & {\bf 320} \\
 \hline
 10& (15\%,~1.6)& (87\%,~0.1)& (100\%,~0)& (100\%,~0)& (100\%,~0)& (100\%,~0)\\
 \hline
 50&  (0\%,~5.4)& (23\%,~1.7) & (81\%,~0.2)& (100\%,~0)& (100\%,~0)& (100\%,~0)\\
 \hline
100&  (0\%,7.8)& (15\%,~2)& (72\%,~0.3)& (97\%,~0)& (100\%,~0)& (100\%,~0)\\
 \hline
  \end{tabular*}
  \caption{Experiments in $B_{20}$.}
  \label{tb:braid20}
  \end{table}

We immediately observe from the data above that:
\begin{itemize}
    \item
the success rate increases as the number of rounds (sample size) increases;
    \item
the success rate decreases as the length of the key increases;
    \item
the success rate increases as the rank of the group increases;
    \item
the average error length decreases as we increase the number of rounds.
\end{itemize}
The first observation is the most interesting since the number of
rounds is one of the main reliability parameters of the protocol,
namely, the soundness error decreases
as $1/2^k$ as the number of rounds $k$ gets larger. But,
at the same time, we observe that security of the
scheme decreases as $k$
increases. The second observation can be interpreted as follows --
the longer the braids are the more difficult it is to compute the
approximation. The third observation is easy to explain. The bigger the
rank of the group the more braid generators commute and the simpler random
braids are.

\subsection{Alternative key generation}
\label{se:special_key_generation}

As we have mentioned in Section \ref{se:zk},
the Sibert et al. scheme, proposed in \cite{SDG}, does not possess perfect zero knowledge
property.
Nevertheless, the authors of \cite{SDG} try to achieve computational zero knowledge
by proposing
a special way of generating public and private information.
They provide some statistical evidence that the scheme can be computationally zero knowledge
if this alternative key generation is used.
In this section we, firstly, outline the proposed key generation method
and, secondly, present actual experiments
supporting our theoretical results even for this special key generation method.

The method of generating of braids in \cite{SDG}
can be translated to the notation of the present paper as follows.
The Prover generates
\begin{itemize}
    \item
nonce elements $r$ as products of $L$ uniformly chosen permutation braids $p_i$ (see \cite{Epstein}) from $B_n$
$$r = p_1\ldots p_{L},$$
in particular, $r$ belongs to the corresponding positive monoid.
    \item
the secret key $s$ as the inverse of a product of $L$ uniformly chosen {\em permutation braids} from $B_n$,
i.e.,
$$s = p_1^{-1}\ldots p_{L}^{-1}.$$
\end{itemize}

We made a very useful observation when doing the experiments with so generated
nonce elements $r$. We observed that the mean-set in this case is often a singleton
set of the form $\{\Delta^k\}$, where $\Delta$ is a half-twist braid and $k\in \MN$. Therefore, to enhance
the performance of Algorithm \ref{al:direct_descend} in step B,
we test not only generators $x\in X^{\pm 1}$, but also $x = \Delta$, and if (in step C)
$\Delta$ minimizes the value of $M_n(gx)$, then we put $x \rightarrow x \Delta$ and return to step B.

In fact it is an interesting question if the uniform distribution on a sphere
in a Garside monoid $G^{+}$ has a singleton mean set $\{\Delta_{G^{+}}^k\}$ for some $k\in \MN$,
where $\Delta_{G^{+}}$ is the Garside element, $\Delta$, in $G^{+}$?
This is clearly true for free abelian monoids.
As we mention above, experiments show that the same can be true in the braid monoid.

Below we present the results of actual experiments done for the group
$B_{10}$. Horizontally we have increasing
number of rounds $k$ from $10$ to $320$ and vertically we have
increasing lengths $L$ (in permutation braids) from $3$ to $10$. Every cell contains a pair $(P\%,E)$
where $P$ is a success rate and $E$ is the average length of the error for the
corresponding pair $(L,k)$ of parameter values.

Since the average Artin length (denoted $L^{\prime}$ in the tables below)
of a permutation braid
on $n$ strands is of order $n^2$, the length
of nonce elements grows very fast with $L$;
it is shown in the leftmost column of the tables
in parentheses. For instance, we can see that for $B_{10}$ the average length
of a product of $L=3$ permutation braids is $81$, the average length
of a product of $L=5$ permutation braids is $138$, etc.
  \begin{table}[ht]\footnotesize
  \centering
  \begin{tabular*}{1\textwidth}%
     {@{\extracolsep{\fill}}|c||c|c|c|c|c|c|}
 \hline
 {\bf L(L')$\backslash$k} &{\bf 10} & {\bf 20} & {\bf 40} & {\bf 80} & {\bf 160} & {\bf 320} \\
 \hline
{\bf  3 (81)} & (0\%,~24.6)  &  (0\%,~22.5)  &  (1\%,~19.6)  & (4\%,~16) & (7\%,~13.1) & (25\%,~12.3)\\
 \hline
{\bf  5 (138)}& (0\%,~46.7)  &  (0\%,~40.9)  &  (0\%,~32.5)  & (2\%,~23.3) & (10\%,~17.6) & (28\%,~14.2) \\
 \hline
{\bf 10 (274)}& (0\%,~110.2) &  (0\%,~102.6) &  (0\%,~103.5) & (0\%,~96.3) & (0\%,~92.7) & (0\%,~87.9) \\
 \hline
  \end{tabular*}
  \caption{Success rate and average length of the error for experiments in $B_{10}$.}
  \label{tb:braid10_pos}
  \end{table}

Again, we observe that success rate increases as we increase the number of rounds, and
the average error length decreases as we increase the number of rounds.

\section{Defending against the attack}
\label{se:defense}

In this section, we describe several principles one can follow
in order to
defend against the mean-set attack presented in this paper or, at least,
to make it computationally infeasible. Defending can be done
through a special choice of the platform group $G$ or a special choice of a distribution $\mu$ on $G$.
Another purpose of this section is to motivate further study of distributions on groups
and computational properties of groups.

\subsection{Groups with no efficiently computable length functions}

One of the main tools in our technique is an efficiently computable function $d_X(\cdot,\cdot)$ on $G$.
To prevent the attacker from computing mean-sets, one can
use a platform group $G$ with a hardly computable length function $d_X(\cdot,\cdot)$
relative to any ``reasonable'' finite generating set $X$. By reasonable generating set
we mean a set, which is small relative to the main security parameter.
Examples of such groups exist.
For instance, length function for any finitely presented group with unsolvable word problem
is not computable. On the other hand, it is hard to work with such groups,
as they do not have efficiently computable normal forms.

A more interesting example is a multiplicative group
of a prime field $\MZ_p^\ast$. The group $\MZ_p^\ast$
is cyclic, i.e., $\MZ_p^\ast = \gp{a}$ for some primitive root $a$
of $p$. It is easy to see that the length of an element
$b\in\MZ_p^\ast$ satisfies
$$
|b| =
\begin{cases}
\log_a b & \mbox{if } \log_a b\le (p-1)/2,\\
p-1-\log_a b & \mbox{otherwirse},\\
\end{cases}
$$
and hence the problem of computing the length of an element
and the discrete logarithm problem  are computationally equivalent.
The discrete logarithm problem is widely believed to be computationally hard
and is used as a basis of security of many cryptographic protocols,
most notably the ElGamal \cite{ElGamal} and Cramer-Shoup \cite{Cramer_Shoup:1998}
cryptosystems.
In other words, $\MZ_p^\ast$ is another example of a group with hardly computable length
function.

\subsection{Systems of probability measures}

Let $G$ be a platform group. Recall that our assumption was that the Prover
uses a fixed distribution on the set of nonce elements, i.e., every element $r_i$
is generated using the same random generator. Instead he can use a sequence
of probability measures $\{\mu_i\}_{i=1}^\infty$, where each measure $\mu_i$,
$i=1,2,\ldots$, is not used
more than once (ever), i.e., every nonce $r_i$, $i=1,2,\ldots$, is generated using a unique
distribution $\{\mu_i\}$. In this case, the attacker does not have theoretical grounds
for working with sampling mean-sets. Nevertheless, it can turn out that the sequence
of random elements $r_1,r_2,\ldots$ can have some other distribution $\mu^\ast$
and the attack will work.
Another difficulty with implementing this idea is that there is no systematic study of distributions
on general finitely generated groups and, in particular, braid groups.
So, it is hard to propose some particular sequence of probability distributions.
Some aspects of defining probability measures on infinite groups are
discussed in \cite{BoMS} and \cite{BMR}.

\subsection{Undefined mean-set}

Another way to foil the attack is to use a distribution $\mu$ on $G$
such that $\ME (\mu)$ is not defined, i.e., the corresponding weight
function is not totally defined.
In that case the assumption of Theorem \ref{th:mean_attack} fails,
and it is easy to see that the sampling weights $M_n(g)$ tend to $\infty$ with probability $1$.
Nevertheless, we still can compare the sampling
weight values, as explained in \cite{Mos:thesis} and \cite{MosUsh:CentralOrder},
where it is shown that the condition of finiteness of $M^{(2)}$
can be relaxed to that of finiteness of $M^{(1)}$. If $M^{(1)}$ is not defined then that
means that the lengths of commitments are too large and are impractical.

\subsection{Large mean-set}
\label{se:large-mean-set}

Also, to foil the attack one can use a distribution $\mu$ on $G$ such that the set $\ME \mu$ is large.
As an example consider an authentication protocol in \cite{Schnorr:1990},
based on the difficulty of computing discrete discrete logarithms in groups of prime order.
The space of nonce elements
in \cite{Schnorr:1990} is an additive group $\MZ_q$ acting by exponentiations on
a bigger group $\MZ_p^\ast$.
It is easy to compute length in $(\MZ_q,+) = \gp{1}$.
But, since the nonce elements $r\in\MZ_q$ are chosen uniformly, it follows that the mean-set is the whole
group $\MZ_q$ (the uniform measure is right-invariant) and in this case
it is impossible to detect the shift $s$ and the mean-set attack fails.
We also refer to \cite{Poupard_Stern:1998} for a modification of
\cite{Schnorr:1990} where nonce elements are not taken modulo $q$
and security proof requires a boundary on the number of times the same key is used.

Now, let $G$ be an infinite group. It is impossible to generate elements of $G$ uniformly,
but one can try to achieve the property described below that can foil the mean-set attack.
Choose a probability measure $\mu$ on $G$ so that the mean-set set $\ME \mu$ is large.
Recall that Algorithm \ref{al:attack} can find up to one element of $G$ minimizing the weight function.
For that it uses Algorithm \ref{al:direct_descend} which randomly (according to some measure $\nu$)
chooses an element of $g\in G$ and then
gradually changes it (descends) to minimize its $M$ value. This way the distribution $\nu$ on the initial choices $g\in G$
defines a distribution $\nu_\mu^\ast$ on the set of local minima of $M$ on $G$. More precisely,
for $g'\in G$,
    $$\nu_\mu^\ast(g') = \mu\{ g\in G \mid \mbox{Algorithm \ref{al:direct_descend} stops with the answer $g'$ on input $g$}\}.$$
Denote by $\mu_s$ the {\em shifted probability measure} on $G$ by an element $s$ defined by
$\mu_s(g) = \mu(s^{-1}g)$. If $S\subseteq G$ is the set of local minima of
the weight function $M$ relative to $\mu$
then the set $sS$ is the set of local minima relative to $\mu_s$. But the distribution
$\nu_{\mu_s}^\ast$ does not have to be induced from $\nu_\mu^\ast$ by the shift $s$, i.e., the equality
$\nu_{\mu_s}^\ast(g) = \nu_{\mu}^\ast(s^{-1}g)$ does not have to hold. In fact, the distributions
$\nu_\mu^\ast$ and $\nu_{\mu_s}^\ast$ can ``favor'' unrelated subsets of $S$ and $sS$ respectively. That would
definitely foil the attack presented in this paper. On the other hand, if $\nu_\mu^\ast$ and $\nu_{\mu_s}^\ast$
are related, then the mean-set attack can still work.

Finally, we want to mention again that probability measures on groups
were not extensively studied and there are no good probability measures known on general groups
and no general methods to construct measures satisfying the desired properties.
Moreover, the problem of making distributions with large mean-sets is very complicated because not every subset of a group $G$
can be realized as a mean-set. See \cite{MosUsh:SLLN1} and \cite{Mos:thesis} for more details.
A number of open questions arise regarding the problems mentioned above, but dealing with
them is beyond the scope of this paper.

\section{Conclusion}
In this paper, we used the probabilistic approach to
analyze the Sibert et al. group-based authentication protocol.
We have proved that the scheme does not meet
necessary security compliances, i.e., it is not computationally zero-knowledge, in practice.
To conduct our analysis, we introduced a new computational problem
for finitely generated groups,
the \textit{shift search problem},
and employed probabilistic tools discussed in \cite{MosUsh:SLLN1} to deal with the problem.
In particular, the concept of the mean-set and
the generalized strong law of large numbers for
random group elements with values in the vertices of
the connected
and locally-finite Cayley graph of
a given infinite finitely-generated group are used.
The rate of success of getting the secret key, as a solution to the
\textit{shift search problem}, has been proved to be linear or exponential depending
on the assumptions one is willing to make.
In addition, we have provided experimental evidence that our approach is practical
and can succeed even for braid groups.
This work shows, among other things, that generalization
of classical probabilistic results to combinatorial objects
can lead to useful applications in
group-based cryptography.

\providecommand{\bysame}{\leavevmode\hbox to3em{\hrulefill}\thinspace}
\providecommand{\MR}{\relax\ifhmode\unskip\space\fi MR }
\providecommand{\MRhref}[2]{%
  \href{http://www.ams.org/mathscinet-getitem?mr=#1}{#2}
}
\providecommand{\href}[2]{#2}

\end{document}